\documentclass[12pt]{article}
\usepackage{amsfonts, amssymb, amsmath, amsthm, xcolor, xspace}

\newcommand{\R}{{\mathbb R}}
\newcommand{\C}{{\mathbb C}}
\newcommand{\N}{{\mathbb N}}

\newcommand{\Cov}{\operatorname{Cov}}
\newcommand{\U}{{\mathcal U}}
\renewcommand{\u}{\mathfrak{u}}
\newcommand{\chiop}{\chi_{\mathrm{op}}}
\newcommand{\st}{\mbox{St}}

\theoremstyle{plain}
\newtheorem{lemma}{Lemma}
\newtheorem{theorem}{Theorem}
\newtheorem{corollary}{Corollary}
\newtheorem{proposition}{Proposition}
\theoremstyle{definition}
\newtheorem{definition}{Definition}
\newtheorem{question}{Question}

\theoremstyle{remark}
\newtheorem{remark}{Remark}
\newcommand{\tempout}[1]{{}}

\title{A Note on Compactness and Clique Size}
\author{David V.\ Feldman \\
{\small University of New Hampshire, Durham} \\[4pt]
  Alexander Wilce \\
  {\small Susquehanna University}
  }
\date{August 8, 2026}

\begin{document}

\maketitle

\begin{abstract} 
Say that a topological space $X$ has {\em finite}, respectively {\em bounded cliques}  
iff every closed, irreflexive binary relation --- equivalently, every closed, loop-free directed graph --- on $X$ has cliques of finite, respectively bounded finite, size. Every compact space has bounded cliques. Having finite cliques implies limit-point compactness, and is implied by $\omega$-limit point compactness (equivalently, countable compactness). Thus, for $T_1$ spaces, having finite cliques is equivalent to countable compactness. Having bounded cliques is strictly weaker than compactness. Indeed, any space $X$ 
such that $X^{\omega}$ is countably %sequentially 
compact has bounded cliques. However, we have found no example of a countably compact space having finite but unbounded cliques. The existence of such a space is the major open problem raised in this note. 
\end{abstract}

\section{Cliques of closed relations}

In the following, $X$ is a topological space 
and $R \subseteq X \times X$ is an {\em orthogonality relation:} a symmetric, irreflexive binary relation, {\em closed in $X^2$} unless otherwise stated. If $X$ is finite, this is just another way to say that $(X,R)$ is a loop-free, undirected graph, but our interest is mainly in the infinite case. The motivating example: $X$ is the unit sphere of $\R^{n}$ or $\C^{n}$, and $R$ is literal orthogonality. By an {\em $R$-clique} in $X$ we mean a pairwise $R$-related subset of $X$. Equivalently, $A \subseteq X$ is a clique iff $A \times A \setminus \Delta_{A} \subseteq R$.\footnote{If $R$ is irreflexive but not necessarily symmetric, we understand cliques of $R$ to be 
those of the symmetrized version, $R \cup R^{t}$, 
which is also irreflexive, and closed if $R$ is 
closed. Thus, we lose nothing by restricting attention 
to the symmetric case.}

\begin{definition}
$R$ has {\em finite cliques} iff all $R$-cliques are finite, and {\em bounded cliques} iff there exists a natural number $n$ with $|A| \leq n$ for all $R$-cliques $A$. \vspace{-.1in}
\end{definition}

One of us stumbled on the following about 20 years ago on the way to some other results { \cite{WiTT, Wi}}:

\begin{lemma}\label{compactbounded}
Suppose $X$ is compact and $R$ is closed in $X^2$. Then $R$ has bounded cliques.
\end{lemma}

\begin{proof}
Let $x \in X$. Since $R$ is irreflexive, $(x,x) \not\in R$. Since $R$ is closed, there exist open sets $U$ and $V$ of $X$ such that $(x,x) \in U \times V$ and $(U \times V) \cap R = \emptyset$. Let $W = U \cap V$, noting that $x \in W$ and $(W \times W) \cap R = \emptyset$. In other words, every point of $X$ has a totally $R$-unrelated open neighborhood. Since $X$ is compact, finitely many of these cover $X$. A clique can contain at most one point from each of these finitely many sets.
\end{proof}

Lemma 1 is to some extent folkloric. 
If $R$ is a nonempty closed orthogonality on a compact 
metric space $X$, then since $R$ is itself compact, 
the metric takes a minimum value $\epsilon > 0$ on $R$ (non-zero since $d$ is zero only on the diagonal, which $R$ misses). Every $R$-clique is thus $\epsilon$-separated (in the non-strict sense: $d(x,y) \geq \epsilon$), and so has size at most the  $\epsilon/2$ packing number of $X$. For compact Hausdorff spaces, Lemma 1 follows from the fact that the complements of orthogonalities are precisely the open symmetric entourages for the unique compatible uniformity (which 
consists of all neighborhoods of the diagonal). Any clique $A$ for the complement of such an entourage $E$ is $E$-discrete ($A \times A \cap E \subseteq \Delta$), and $E$-discrete sets, for a fixed entourage $E$, in a totally bounded uniform space are boundedly finite.\footnote{This is easy to show, but not easy to find in the textbook literature, so we supply a 
short proof. Suppose $A$ is $E$-discrete. Choose a smaller symmetric entourage $D$ with $D \circ D \subseteq E$: since $X$ is totally bounded, there is some finite set $F \subseteq X$ with $X \, = \, D[F] \, := \, \bigcup_{a \in F} D[a]$. If $x,y \in A \cap D[a]$, then $(a,x), (a,y) \in D$; as $D = D^{-1}$, we also have $(x,a) \in D$, whence $(x,y) \in D \circ D \subseteq E$. But then, as $A$ is $E$-discrete, $x = y$. Thus, each set $D[a]$ contains at most one point of $A$, whence, $|A| \leq |F|$.}

Lemma 1 also appears in the literature on topological packing graphs \cite{dLV} (where the statement is that compactness of the vertex set implies a finite independence number). What the proof of Lemma 1 adds is the observation that no separation axioms, and thus, no 
uniform or metric structure, is required. \\

\begin{remark}\label{lemmaremarks} 
(i) Lemma 1 proves more than boundedness. Defining the {\em open chromatic number} $\chiop(R)$ as the least cardinality of a cover of $X$ by $R$-free open sets, the proof shows $\chiop(R)$ is finite when $X$ is compact, and evidently every clique $A$ satisfies $|A| \le \chiop(R)$. We return to this in \S 5. 

(ii) In descriptive graph combinatorics, a 
closed graph is usually understood to be one 
closed in $X^2 \setminus \Delta$ (see, e.g., 
\cite{AZ}). Our requirement, 
closure in $X^2$, is much stronger than this. 
In some measure, this note discusses what  
this additional strength buys.
\end{remark}

\begin{definition}  A topological space $X$ has  
{\em finite}, respectively {\em bounded cliques}, 
iff every closed orthogonality relation $R$ on $X$ has finite, respectively bounded, cliques.   \end{definition} 

It will sometimes also be convenient to refer to the property of having finite, resp. bounded, cliques as the {\em finite clique property} (FCP) or the {\em bounded clique property} (BCP). In this lingo, Lemma 1 says that compact spaces have the BCP.

We recall that an {\em $\omega$-limit point} of a set $A \subseteq X$ is a point $x \in X$,
every open neighborhood of which has infinite intersection with $A$. A space is {\em $\omega$-limit point compact} iff every infinite set has an $\omega$-limit point. In a $T_1$ space, ordinary limit points are also $\omega$-limit points, and $\omega$-limit point compactness is equivalent to ordinary 
limit-point compactness. For arbitrary spaces, 
it is equivalent to countable compactness.
\footnote{Again, easy to show, but not easy to find, 
so again we supply a short proof. If $X$ is countably compact and $A \subseteq X$ is infinite, we can choose 
a countable subset $\{a_n \mid n \in \N\}$ of $A$. 
Let $F_n$ be the closure of $\{a_k \mid k \geq n\}$, 
and note that these are nested, so countable compactness yields a point $a \in  \bigcap_{n} F_n$, and this is an $\omega$-limit point of $A$. Conversely, suppose $X$ is not countably compact, and let $\{U_n\}$ be a countable open cover with no finite subcover. Then $\{U_n\}$ is infinite. For each $n > 1$, let $V_{n} :=  \bigcup_{k =1 }^{n} U_k$. Then we have an increasing sequence of proper open sets with union $X$. Select $x_n \not \in V_n$. Nestedness implies $x_i \not\in V_n$ whenever $i \geq n$. Since every point of $X$ lies in some $V_N$, and hence in every $V_n$ with $n \geq N$, no point can equal $x_n$ for arbitrarily large $n$; thus $A := \{x_n\}$ is infinite. Moreover, $V_n \cap A \subseteq \{x_1, \ldots, x_{n-1}\}$, a finite set. As every point of $X$ has some $V_n$ as a neighborhood, $A$ has no $\omega$-limit point.}

\begin{lemma}\label{finitecc} For any $X$, 
\begin{itemize} 
\item[(a)] If $X$ has finite cliques, it is limit point compact. 
\item[(b)] If $X$ is $\omega$-limit point compact (equivalently, countably compact), it has finite cliques.
\end{itemize} 
\vspace{-.1in}
Hence, a $T_1$ space has finite cliques iff it is 
limit-point compact iff it is countably compact. 
\end{lemma}

\begin{proof}
(a) Suppose $X$ contains an infinite set $A$ with no limit point. Then $A$ is closed and discrete in $X$. Let $R := A \times A \setminus \Delta_{A}$; we claim $R$ is closed. Since $A$ is closed, $\overline{R} \subseteq A \times A$. Suppose $(x,y) \in \overline{R}$. If $x \neq y$ then $(x,y) \in A \times A \setminus \Delta_A = R$. If $x = y = a \in A$, then since $a$ is not a limit point of $A$ there is an open $U \ni a$ with $U \cap A = \{a\}$, so $(U \times U) \cap (A \times A) = \{(a,a)\}$, which misses $R$; hence $(a,a) \notin \overline{R}$. So $R$ is a closed orthogonality relation with the infinite clique $A$. 

For (b), suppose $R$ has an infinite clique $A$, so that $A \times A \setminus \Delta_{A} \subseteq R$. Let $x$ be an $\omega$-limit point of $A$. Then for any open set $U \subseteq X$ with $x \in U$, $U \cap A$ contains infinitely many points of $A$ other than $x$.  
%\footnote{$U$ contains some $y_{1} \in A \setminus \{x\}$. Since $X$ is $T_{1}$, $\{y_{1}\}$ is closed, so $U \setminus \{y_{1}\}$ is also an open neighborhood of $x$, so it, too, contains a point $y_{2}$ of $A \setminus \{x\}$, and so on.} 
Now $(U \times U) \cap (A \times A)$ contains infinitely many pairs $(y,z)$ with $y \neq z$, hence infinitely many points of $A \times A \setminus \Delta_{A}$. Thus $(x,x)$ is a limit point of $A \times A \setminus \Delta_{A}$, hence of $R$; since $R$ is closed, $(x,x) \in R$, contradicting irreflexivity.
\end{proof}

For metrizable spaces, countable compactness is equivalent to compactness, so we have the

\begin{corollary}\label{metrizable}
Let $X$ be metrizable with finite cliques. Then $X$ is compact, and in particular has bounded cliques.
\end{corollary}

For $T_1$ spaces, our picture at this point is as follows:
\begin{eqnarray*}
\mbox{$X$ compact} & \Rightarrow & \mbox{$X$ has bounded cliques} \\
& \Rightarrow & \mbox{$X$ has finite cliques} \; \Leftrightarrow \;
X \mbox{ is countably compact}.
\end{eqnarray*}
The first implication is strict. { In fact, 
$\omega_1$ is an example of a non-compact space with bounded cliques. This follows from Theorem \ref{omegapower} below.}

Whether the second implication is strict is the {\bf main open question} of this note:

\begin{question}\label{mainq}
Is there a countably compact $T_1$ (or Hausdorff) space with a closed orthogonality relation having finite but unbounded cliques?\\
\end{question}

\begin{remark}\label{nonT1} It is not hard to find non-$T_1$ lp-compact spaces having infinite cliques (consider, e.g., $\N \times \{0,1\}$ with $\N$ discrete and $\{0,1\}$ indiscrete, and 
consider the relation of being in different fibres over $\N$). However, we do not know whether, in the general, non-$T_1$ setting, $\omega$-lp compactness is necessary for finite cliques.  The sufficiency of $\omega$-lp compactness gives us a non-$T_1$ version of Question 1: 
is there an $\omega$-lp compact space hosting a closed, irreflexive relation with cliques of arbitrary 
size? \\
\end{remark}

\begin{remark}[A dichotomy]\label{dichotomy}
{ Lemma \ref{finitecc}} also leads to an interesting dichotomy. Either Question \ref{mainq} has a positive answer, or  BCP coincides, for $T_1$ spaces, with countable compactness. In the latter case, BCP is {\em not finitely productive} in ZFC: the construction of \cite{HvMRS} yields countably compact groups $G, H$ with $G \times H$ not countably compact, hence (on that horn) with $G \times H$ failing even finite cliques while $G$ and $H$ have bounded cliques. 
\end{remark}

\section{A Galois Connection} 

Both the FCP and BCP can be reexpressed as covering properties. Let $\U$ be an open cover of $X$. Define a binary relation
\[x \perp_{\U} y \ \Leftrightarrow \ \neg \exists U \in \U \ \ x,y \in U.\]
Then $\perp_{\U}$ is a closed (possibly empty) orthogonality relation. Call a set $A \subseteq X$ {\em $\U$-separated} if no two distinct points of $A$ lie in a common member of $\U$; the $\perp_{\U}$-cliques are exactly the $\U$-separated sets.

Let $R$ be a closed orthogonality relation. Define a cover
\[\Cov(R) \ := \ \{ U \subseteq X \ | \ U \mbox{ open}, \ (U \times U) \cap R = \emptyset\},\]
the collection of totally non-$R$-related opens. (That this is a cover is shown as in the proof of Lemma \ref{compactbounded}.)

\begin{proposition}\label{galois}
$(\perp, \Cov)$ is an antitone Galois connection between open covers and closed orthogonalities, ordered by inclusion: for every open cover $\U$ and closed orthogonality $R$,
\[ R \subseteq \perp_{\U} \iff \U \subseteq \Cov(R). \]
\end{proposition}

\begin{proof}
Both sides say: no member of $\U$ contains an $R$-related pair.
\end{proof}

Thus $R$ is Galois-closed iff $R = \perp_{\Cov(R)}$ iff $R = \perp_{\U}$ for some $\U$, and $\U$ is Galois-closed iff $\U = \Cov(\perp_{\U})$ iff $\U = \Cov(R)$ for some closed orthogonality $R$. Since $R \subseteq \perp_{\Cov(R)}$ always, and enlarging a relation only enlarges its family of cliques, we get the useful

\begin{corollary}[Reduction to covers]\label{covers}
The suprema of clique sizes over all closed orthogonalities and over all relations of the form $\perp_{\U}$ coincide. Hence, $X$ has finite cliques iff for every open cover $\U$, every $\U$-separated set is finite, and $X$ has bounded cliques iff for every open cover $\U$ there is some $n \in \N$ with all $\U$-separated sets having size $\leq n$.  
\end{corollary}

This re-proves Lemma \ref{compactbounded}: if $X$ is compact, some finite subfamily $U_1, \dots, U_n$ of $\U$ covers $X$, and a $\U$-separated set meets each $U_i$ at most once.

Corollary \ref{covers} brings us into contact 
with star-covering properties \cite{Cao,vDRRT}. 
A space $X$ is {\em strongly star-compact} iff for every open cover ${\mathcal U}$, there is a finite set $F$ with \[\st(F,{\mathcal U}) := \bigcup \{ U \in {\mathcal U} \mid F \cap U \not = \emptyset\} = X.\]

If $F$ is $\mathcal U$-separated and 
$y \not \in \st(F,{\mathcal U}) = \bigcup\{ U \in {\mathcal U} \mid F \cap U \not = \emptyset\}$, then 
$F \cup \{y\}$ is again $\mathcal U$-separated. Hence, 
if $F$ is a maximal $\mathcal U$-separated set (these 
exist by Zorn), we must have $\st(F,{\mathcal U}) = X$. Thus, since a ${\mathcal U}$-separated set is a clique 
of $\perp_{\mathcal U}$, we have 

\begin{corollary}\label{starcompact} A space with finite cliques is strongly star-compact.
\end{corollary} 

For $T_1$ spaces, this reduces to the known 
fact that countable compactness implies strong star-compactness. The converse is false (e.g., $\R$ with the co-countable topology is strongly star-compact but 
not countably compact), and thus, so is the converse to Corollary \ref{starcompact}. 

We have found no named star-covering property matching ``bounded cliques,'' which we take as evidence that the main question is new.

\section{Cliques and Powers}

The following theorem supplies a wealth of 
examples of non-compact spaces with bounded cliques.

\begin{theorem}\label{omegapower}
If $X^{\omega}$ is countably compact, then $X$ has bounded cliques.
\end{theorem}

\begin{proof} 
The hypothesis implies that $X$ itself is countably compact (being the continuous image of $X^{\omega}$ under a projection), so Lemma 2(b) applies: no  closed orthogonality on $X$ has an infinite clique.  
Suppose some closed orthogonality $R$ on $X$ admits  arbitrarily large finite cliques $C_n = \{x^n_1, \dots, x^n_n\}$ for every $n$.
 Define $y_n \in X^{\omega}$ by $(y_n)_i = x^n_{\min(i,n)}$. (Note that these 
 are distinct, as $y_{n}$ has $n$ distinct initial 
 entries, after which all entries are constant.) 
Let $q = (q_1, q_2, \dots)$ 
be an $\omega$-limit point of $\{y_n \mid n \geq 1\}$.
Fix $i \neq j$ and open neighborhoods $U$ and $V$ of $q_i$ and $q_j$, respectively.
The basic open set $\{z \in X^\omega : z_i \in U, \, z_j \in V\}$ is a neighborhood of $q$, hence contains $y_n$ for infinitely many $n$; choose such an $n \geq \max(i,j)$. Then $(y_n)_i = x^n_i$ and $(y_n)_j = x^n_j$ are distinct members of the clique $C_n$, so
\[ (x^n_i, x^n_j) \in R \cap (U \times V). \]
As $U, V$ were arbitrary, $(q_i, q_j) \in \overline{R} = R$; irreflexivity then gives $q_i \neq q_j$. So $\{q_i : i \geq 1\}$ is an infinite $R$-clique, 
a contradiction.
\end{proof}

\begin{corollary}\label{suffconds}
Each of the following implies that $X$ has bounded cliques:
\begin{enumerate}
\item[(a)] $X$ is sequentially compact;
\item[(b)] $X$ is $\omega$-bounded (every countable subset has compact closure);
\item[(c)] $X$ is $\u$-compact for some free ultrafilter $\u$ on $\omega$.
\end{enumerate}
\end{corollary}

\begin{proof}
In each case $X^\omega$ is countably compact. (a) Sequential compactness is countably productive (diagonalize), and sequentially compact spaces are countably compact. (b) Let  $D = \{z^{n}\} \subseteq X^\omega$: by $\omega$-boundedness, the sets $K_i = \overline{\pi_{i}(D)}$ are compact in $X$, 
and thus $D \subseteq K:= \Pi_i K_i$, a compact 
subspace of $X^{\omega}$.  In our general setting, 
we cannot conclude that $K$ is closed; however, 
consider the sets $\overline{\{z^n \mid n \geq N\}} \subseteq K$ (the closure taken in $K$): these have the finite intersection property, so a cluster point for $D$ exists in $K$, and hence, in $X^{\omega}$. Thus, $X^{\omega}$ is countably compact. 
(c) For any free ultrafilter $\u$, $\u$-compactness is fully productive { \cite[Theorem 4.2]{Bern}} and implies countable compactness. 
\end{proof}

\begin{remark} 
It is consistent with ZFC that the hypothesis of Theorem 1 is strictly weaker than that of 
Corollary \ref{suffconds} (c).  Assuming the existence of $2^{\mathfrak c}$ pairwise incomparable selective ultrafilters, together with $2^{<2^{\mathfrak c}} = 2^{\mathfrak c}$ (both follow, e.g., from GCH), Tomita \cite{To}  showed that for every cardinal $\kappa \leq 2^{\mathfrak c}$, there exists a topological group $G$ with $G^{\gamma}$ countably compact for $\gamma < \kappa$, but $G^{\kappa}$ not. Taking $\kappa = 2^{\mathfrak c}$, we have $G^{\omega}$ countably compact but $G^{2^{\mathfrak c}}$ not. The 
Ginsburg-Saks theorem \cite{GS} tells us that 
if $G$ is $\u$-compact for any free ultrafilter, 
all cardinal powers of $G$ are countably compact; 
hence, $G$ is not $\u$-compact for any free ultrafilter. Theorem 1 applies to such a space, but Corollary \ref{suffconds} (c) does not. (As far as we know, the existence of a space $X$ with $X^{\omega}$ countably compact but 
some higher power not so, is open in ZFC.)
\end{remark} 

\begin{remark}\label{omegaone}
Theorem 1 implies that any sequentially compact but 
non-compact space $X$ has bounded cliques. The 
classic example is $\omega_1$. 
In this case, there is also a direct argument 
that bounds cliques explicitly. Given a closed orthogonality $R$ on $\omega_1$, there is $\alpha < \omega_1$ with
\[ R \cap \big( (\alpha, \omega_1) \times (\alpha, \omega_1) \big) = \emptyset. \]
(Otherwise recursively choose $(a_n, b_n) \in R$ with $\min(a_{n+1}, b_{n+1}) > \max(a_n, b_n)$; with $\gamma := \sup_n a_n = \sup_n b_n$ we get $(a_n,b_n) \to (\gamma,\gamma) \in \overline{R} = R$, contradicting irreflexivity.) A clique therefore has at most one member above $\alpha$, and its remainder is a clique of the restriction of $R$ to the compact space $[0,\alpha]$, which is bounded by Lemma \ref{compactbounded}. Note the displayed statement is best possible: $R$ itself need not be order-bounded, as the closed cofinal orthogonality
\[ R \;=\; \big( \{0\} \times (0,\omega_1) \big) \cup \big( (0,\omega_1) \times \{0\} \big) \]
shows (its cliques have size $2$; closedness at $(0,0)$ holds because $\{0\}$ is clopen).
%\end{example}
\end{remark}

\section{The structure of the problem}

In this section, we establish some further 
characteristics of spaces with, and countably compact 
spaces without, bounded cliques. These considerably focus the possible lines of attack on Question 1.

\begin{proposition}[Preservation]\label{preserve}
(i) Bounded cliques is closed-hereditary. (ii) Bounded cliques is preserved by continuous surjections.
\end{proposition}

\begin{proof}
(i) If $C \subseteq X$ is closed and $R$ is a closed orthogonality on $C$, then $R$ is closed in $X^2$ and is an orthogonality relation on $X$ with the same cliques.

(ii) Let $f : X \twoheadrightarrow Y$ be a continuous surjection and $R$ a closed orthogonality on $Y$ with cliques of size $\geq n$ for every $n$. The pullback $R' := (f \times f)^{-1}(R)$ is closed, symmetric, and irreflexive (as $R$ is). If $A \subseteq Y$ is an $R$-clique, choose $x_a \in f^{-1}(a)$ for each $a \in A$; the $x_a$ are distinct and pairwise $R'$-related. So $R'$ has unbounded cliques and $X$ fails bounded cliques.
\end{proof}

\begin{proposition}[Separable reflection]\label{separable}
If some countably compact $T_1$ space carries a closed orthogonality with finite, unbounded cliques, then some {\em separable} such space does.
\end{proposition}

\begin{proof}
Let $R$ on $X$ have cliques $C_n$ with $|C_n| = n$, let $D = \bigcup_n C_n$, and let $Y = \overline{D}$. Then $Y$ is separable, closed in $X$ --- hence countably compact --- and $R \cap Y^2$ is a closed orthogonality on $Y$ with the same unbounded cliques.
\end{proof}

\begin{corollary}\label{metacompact}
A metacompact (in particular, paracompact, in particular, metrizable) $T_1$ space has finite cliques iff it has bounded cliques iff it is compact. 
\end{corollary}

\begin{proof}
Finite cliques gives countable compactness (Lemma \ref{finitecc}); countably compact metacompact spaces are compact \cite{AD}; compactness gives bounded cliques (Lemma \ref{compactbounded}).
\end{proof}

Taken together, Theorem \ref{omegapower}, Corollary \ref{suffconds}, Proposition~\ref{separable}, and Corollary \ref{metacompact} severely constrain the shape of any witness $X$ to a positive answer to Question \ref{mainq}: it may be taken to be separable, and must be countably compact; 
 but it must also be non-compact, non-sequentially-compact, non-$\omega$-bounded, non-metacompact, not $\u$-compact for any free ultrafilter $\u$, and --- most tellingly --- must have $X^{\omega}$ {\em not} countably compact.

The following construction suggests a possible 
candidate for such a witness.  Partition $\omega$ into sets $A_n$ with $|A_n| = n$, and let $R_0$ join distinct points of a common cell. Call $S \subseteq \omega$ a {\em partial section} if $|S \cap A_n| \leq 1$ for all $n$. For a free ultrafilter $\mathfrak{u}$, a direct computation with basic neighborhoods gives
\[ (\u,\u) \in \overline{R_0} \; \mbox{ (closure in } (\beta\omega)^2 \mbox{)} \iff \mbox{no member of } 
\u \mbox{ is a partial section}. \]
(Note here that if some $B \in \u$ meets every cell at most $k$ times, then $B$ splits into at most $k$ partial sections, one of which lies in $\u$; so bounded transversals already witness membership in the good set.) 
Let $G := \{ \u \in \omega^* : \u \mbox{ contains a partial section} \}$. Then $G$ is open (a union of basic clopen sets $\widehat{S} = \{ \u \mid S \in \u\}$, $S$ a partial section) and dense (every infinite $B \subseteq \omega$ meets infinitely many cells, since cells are finite, { and} hence contains an infinite partial section), while $\omega^* \setminus G$ is nonempty (complements of finitely many partial sections $S_1, \dots, S_k$ retain at least $n - k$ points of $A_n$ { for $n > k$}, so these complements have the finite intersection property) and closed nowhere dense. Consequently:

{\em If there is a dense countably compact $Y \subseteq \omega^*$ with $Y \subseteq G$, then $X := \omega \cup Y$, with $R := \overline{R_0} \cap X^2$, answers Question \ref{mainq} positively}: $R$ is a closed orthogonality on $X$ (irreflexivity at points of $Y$ is exactly $Y \subseteq G$; points of $\omega$ are isolated), the cells $A_n$ are cliques of every size, and $X$ is countably compact. For density, it in fact suffices that $Y$ meet $\widehat{S}$ for every infinite partial section $S$, since every basic clopen $\widehat{B}$ contains such an $\widehat{S}$ --- so the required $Y$ lives, so to speak, entirely on the clopen skeleton of $G$.\\

\section{Chromatic Questions} 

Define $\chiop(R)$ as in Remark \ref{lemmaremarks}.  
Let us say that $X$ has {\em finite open chromatic number} if $\chiop(R) < \infty$ for every closed orthogonality $R$. Then
\[ \mbox{compact} \Rightarrow \mbox{finite } \chiop \Rightarrow \mbox{bounded cliques}\]
The first implication is strict: see the 
discussion of $\omega_1$ in Remark \ref{omegaone} in \S3; split at the $\alpha$ of the Remark and use the compactness of $[0,\alpha]$,  which is also clopen in $\omega_1$, so that a finite cover of $[0,\alpha]$ by $R$-free relatively open sets is already a family of $R$-free open subsets of $\omega_1$. The remaining piece $(\alpha,\omega_1)$ is itself open and $R$-free, and so serves as one further color class.

\begin{question} 
Is the second implication strict? 
\end{question}

For $X = S^{n-1} \subseteq \R^n$ with literal orthogonality, cliques are orthonormal systems, so the clique number is exactly $n$. The chromatic side has a substantial literature, rooted in Kochen--Specker \cite{KS} and running through Simmons \cite{Si} and  Lov\'asz \cite{Lo}: an {\em orthogonal coloring} of $S^2$ is a partition into parts containing no orthogonal pair. Godsil--Zaks \cite{GZ} proved that the chromatic number of the orthogonality graph of $S^2$ is exactly $4$. 

These colorings are unrestricted, but our $\chiop$ demands {\em open} color classes, a stronger constraint. The eight open octants are orthogonality-free and cover 
$S^2 \setminus C$ where 
$C$ is the union of the three coordinate great circles. 
Rotations take orthogonality-free opens to orthogonality-free opens. Choosing  $\rho_2$, $\rho_3$ with $C \cap \rho_2(C) \cap \rho_3(C) = \emptyset$ (
a generic $\rho_2$ gives $C \cap \rho_2(C)$ finite; then choose a $\rho_3$ with $\rho_{3}(C)$ avoiding these finitely many points) gives us 
\[ 4 \; = \; \chi(S^2, \perp) \; \leq \; \chiop(S^2, \perp) \; \leq \; 24.\]

We have not found the open chromatic number of literal orthogonality on $S^2$ treated in the literature, 
which suggests:
%\end{remark}

\begin{question}\label{sphereq}
What is the least number of open orthogonality-free sets covering $S^2$? More generally, determine or estimate $\chiop(S^{n-1}, \perp)$.\\
\end{question}

\begin{remark} Holmsen and Lee \cite{HL} showed that an orthogonal $4$-coloring of $S^2$ all of whose  color classes have non-empty interior is necessarily an  octahedral one, with the interiors of the color-classes the unions of antipodal open octants. This suggests that $\chiop(S^2) \geq 5$, but we have not been able to verify this. 
\end{remark} 

{\bf Machine Assistance} This paper began as an informal note (essentially comprising sections 1 and 2 and a slightly mangled version of the $\omega_1$ example), shared between the authors in 2024. Recent dialogues with  Anthropic's Claude (July and August, 2026) pointed us to the relevant literature, and supplied Corollary 3 and the results of Sections 3, 4 and 5 (in particular, Theorem 1 and its Corollary), which we have checked carefully, and revised for readability.

{\sloppy
{\bf Verification} The mathematical content of this note has been formalized in Lean 4 against
Mathlib (toolchain \texttt{v4.28.0}) and machine-checked, with no {\tt sorry}s and no axioms
beyond Lean's standard three ({\tt propext}, {\tt Classical.choice}, {\tt Quot.sound}):
Lemmas \ref{compactbounded} and \ref{finitecc} with their footnotes, Corollary \ref{metrizable},
Proposition \ref{galois} with Corollary \ref{covers}, Corollary \ref{starcompact} and the failure
of its converse, the dichotomy of Remark \ref{dichotomy}, Theorem \ref{omegapower} and Corollary \ref{suffconds},
the $\omega_1$ example of Remark \ref{omegaone}, Propositions \ref{preserve} and \ref{separable},
Corollary \ref{metacompact}, the profile of a witness, the $\beta\omega$ construction of \S 4
together with its conditional consequence, the implications of \S 5, the clique number of
$S^{n-1}$, and the bound $\chiop(S^2,\perp) \leq 24$. Question \ref{mainq}, the question of \S 5
as to strictness of the second implication, and the two subsidiary questions of
Remark \ref{nonT1}, are stated formally and left unresolved; Question \ref{sphereq}, which asks
for a value rather than a truth value, is not formalized. Results quoted from the
literature are not formalized; in particular the lower bound $\chi(S^2,\perp) = 4$ is not, so the
formal sphere statement is $3 \leq \chiop(S^2,\perp) \leq 24$, and the $\leq 24$ bound is
obtained from three explicit orthonormal frames of $\R^3$ in place of the genericity argument
given above. The development is available at
\texttt{https://github.com/\allowbreak DavidVFeldman/\allowbreak compactness-clique-size},
archived at \texttt{https://doi.org/\allowbreak 10.5281/zenodo.21910510}.\par}

\end{document}